\documentclass[10pt]{article}

\usepackage{amsfonts, epsfig, amsmath, amssymb, color,mathabx,amsthm}
\usepackage{mathrsfs}
\usepackage[english]{babel}

\usepackage{hyperref}

\newcommand{\E}{\mathbf{E}}
\renewcommand{\P}{\mathbf{P}}

\renewcommand {\epsilon}{\varepsilon}

\theoremstyle{plain}
\newtheorem{thm}{Theorem}[section]
\newtheorem{lem}[thm]{Lemma}

\theoremstyle{definition}
\newtheorem{rem}[thm]{Remark}

\DeclareMathSymbol{\ophi}{\mathalpha}{letters}{"1E}

\newcommand{\e}{\varepsilon}

\renewcommand{\phi}{\varphi}

\newcommand{\be}{\begin{equation}}
\newcommand{\ee}{\end{equation}}
\newcommand{\ben}{\begin{equation*}}
\newcommand{\een}{\end{equation*}}

\newcommand{\ba}{\begin{equation}\begin{aligned}}
\newcommand{\ea}{\end{aligned}\end{equation}}
\newcommand{\ban}{\begin{equation*}\begin{aligned}}
\newcommand{\ean}{\end{aligned}\end{equation*}}

\renewcommand{\i}{\mathrm{i}}
\newcommand{\ex}{\mathrm{e}}
\newcommand{\di}{\mathrm{d}}

\newcommand{\rB}{\mathscr{B}}

\newcommand{\rF}{\mathscr{F}}

\newcommand{\bR}{\mathbb{R}}

\let\oldmarginpar\marginpar
\renewcommand{\marginpar}[1]{\oldmarginpar{\scriptsize\texttt{\color{red}{#1}}}}

\allowdisplaybreaks[4]

\begin{document}
\title{Strong uniform Wong--Zakai approximations\\ of L\'evy-driven Marcus SDEs}



\author{Ilya Pavlyukevich\thanks{Institute of Mathematics, Friedrich Schiller University Jena, Ernst--Abbe--Platz 2, 
07743 Jena, Germany; \texttt{ilya.pavlyukevich@uni-jena.de}}\,\,  
and 
Sooppawat Thipyarat\thanks{Institute for Innovative Learning, Mahidol University, Nakhon Pathom, Thailand (\texttt{sooppawat.thi@mahidol.ac.th})}
}

\maketitle

\begin{abstract}
For a solution $X$ of a L\'evy-driven $d$-dimensional Marcus (canonical) stochastic differential equation,
we show that the Wong--Zakai type approximation scheme $X^h$
has a strong convergence of order $\frac12$: for each $T\in [0,\infty)$ and all $x\in\bR^d$ we have
$$
\E \sup_{kh\leq T}|X_{kh}(x)-X^h_{kh}(x)|\leq  C h^{\frac{1}{2}}(1+|x|),\quad h\to 0.
$$
We also determine the rate of the locally uniform strong convergence: for each $N\in(0,\infty)$ we have
$$
\E\sup_{|x|\leq N}\sup_{kh\leq T}|X_{kh}(x)-X^h_{kh}(x)|\leq C h^{\frac{1-\e}{4d}},\quad h\to 0.
$$
\end{abstract}


\smallskip


\numberwithin{equation}{section}

\section{Introduction}

In 1965, Wong and Zakai \cite{WongZakai-65,WongZakai-65b} posed the following question.
Let $W=(W_t)_{t\geq 0}$ be a standard (one-dimensional) Brownian motion, and let $W^h=(W^h_t)_{t\geq 0}$ be a sequence
of polygonal approximations of $W$, i.e., a sequence of random continuous functions of bounded variation
with piece-wise continuous derivatives that are given by the formula
\begin{equation}
\label{e:polW}
 W^h_t=W_{kh}+\frac{t-kh}{h}\Big(W_{(k+1)h}-W_{kh}\Big),\quad t\in[kh,(k+1)h),\quad k\in\mathbb N_0,\\
\end{equation}
for $h\in(0,1]$. It is clear, that the approximations $W^h$ converge to $W$ uniformly on each compact interval
$[0,T]$ with probability 1 as $h\to 0$.
Let $a$ and $b$ be Lipschitz continuous functions.
Then, for each $h\in(0,1]$, the random ordinary differential equation (ODE)
$\dot X^h=a(X^h) + b(X^h)\dot W^h$ has a unique
solution on $[0,T]$ with probability 1. How can one describe the limit of $X^h$ as $h\to 0$?

It turns out that the limit process $X^\circ$ exists as a uniform a.s.\ limit of $X^h$ and it satisfies the \emph{Stratonovich}
SDE $\di X^\circ=a(X^\circ)\,\di t+ b(X^\circ)\circ \di W$.
This type of approximations, called \emph{the Wong--Zakai approximations},
was studied by many authors, see, e.g., the review paper by Trawdowska
\cite{Twardowska-96}.
The classical treatment of convergence can be found in Chapter VI Section 7 of Ikeda and Watanabe \cite{IW89}.
The following results hold true.
\begin{thm}[\cite{IW89}, Chapter VI, Theorem 7.2, Remark 7.2 and Theorem 7.3]
\label{t:11}
Let $a\in C^2_b(\bR^d,\bR^d)$ and $b\in C^2_b(\bR^d,\bR^{d\times m})$. Then for every $T\in[0,\infty)$ we have
\begin{equation}
\label{e:stra}
\lim_{h\to 0}\sup_{x\in\bR^d}\E \sup_{t\in[0,T]}|X^{\circ}_t(x) - X^{h}_t(x) |^2=0.
\end{equation}
Moreover, for every $T\in[0,\infty)$ and $N\in(0,\infty)$ and $p\in[2,\infty)$ we have
\begin{equation}
\lim_{h\to 0}\E \sup_{|x|\leq N}\sup_{t\in[0,T]}|X^{\circ}_t(x) - X^{h}_t(x) |^p=0.
\end{equation}
\end{thm}

The speed of convergence in \eqref{e:stra} depends on the
smoothness properties of the coefficients $a$ and $b$. For instance,
Gy{\"o}ngy and Michaletzky \cite{GyongyM-04} proved the following theorem.
\begin{thm}[\cite{GyongyM-04}, Remark 2.6 and Theorem 2.8]
Let $a$ be bounded and Lipschitz-continuous, $b$ be bounded and have a bounded and Lipschitz-continuous derivative, and
let $X_0=\xi$ be a random variable independent of $W$.
Then for every $T>0$ there is $C>0$ such that for $h>0$ small enough
$$
\sup_{t\in[0,T]}\E |X^{\circ}_t - X^{h}_t |^2\leq C h^\frac{1}{2}
$$
\end{thm}
Brze\'zniak and Flandoli \cite{BrzFla-95} obtained the following result.
\begin{thm}[Theorem 2.4 in \cite{BrzFla-95}]
Let $a\in C^1_b(\bR^d,\bR^{d})$ and $b\in C^2_b(\bR^d,\bR^{d\times m})$.
Then for every $T\in [0,\infty)$ and $p\in[1,\infty)$ there is $C>0$ such that for $x\in\bR^d$ and $h\in(0,1]$
\begin{displaymath}
\sup_{t\in[0,T]}\E |X^{\circ}_t - X^{h}_t |^{2p}\leq C h^{p-1}.
\end{displaymath}
\end{thm}
Shmatkov \cite{shmatkov2006rate} obtained the following a.s.\ convergence.
\begin{thm}[Theorem 2.3.2 in \cite{shmatkov2006rate}]
\label{t:14}
Let $a$ be Lipschitz continuous and $b\in C^3(\bR^d,\bR^{d\times m})$ with bounded derivatives.
Let $X_0=\xi$ be a random variable independent of $W$.
Then for every $T\in(0,1/2)$ and any $\gamma\in(0,1/2)$ there is a finite random variable $C=C(\omega)$ such that
for $h>0$ small enough
$$
\sup_{t\in[0,T]} |X^{\circ}_t(\omega) - X^{h}_t(\omega) |\leq C(\omega) h^{\gamma}.
$$
\end{thm}

The Wong--Zakai question can be also posed for more general SDEs driven by L\'evy processes (or semimartingales with jumps).
The corresponding limit process turn out to be a solution of the so-called  Marcus (canonical) SDE.

Marcus (canonical) SDEs $\di X=a(X)\,\di t + b(X)\circ \di W + c(X)\diamond\di Z$ driven by a Brownian motion $W$ and a
pure jump L\'evy process $Z$
have been introduced by S.\ Marcus in \cite{Marcus-78,Marcus-81} as generalizations of Stratonovich SDEs.
If the driving process $Z$ has a jump $\Delta Z_\tau$ at time instant
$\tau$, the solution $X$ makes a jump of the size $\varphi^{\Delta Z_\tau}(1;X_{\tau-})-X_{\tau-}$
where $\varphi$ is a solution of a supplementary
nonlinear
ODE $\dot \varphi^z(u;x)=c(\varphi^z(u;x))z$ with the initial condition $\varphi^z(0;x)=x$
that is defined with respect to a fictitious time $u\in[0,1]$.
The interpretation of this construction is the following:
the jump of the process $X$ is a mathematical idealization of the infinitely
fast motion along the integral curve of the vector field $c(\cdot)\Delta Z_\tau$.

Similarly to Stratonovich SDEs, Marcus equations formally follow the rules of the conventional calculus.
The It\^o formula for solutions of Marcus SDEs formally coincides
with the Newton--Leibnitz change of variables formula. The application area of Marcus SDEs also verlaps with that
of Stratonovich SDEs, see, e.g.\ \cite{ChePav-14,HarPav-23}.
First, they are indispensable for describing random dynamics with jumps on manifolds. Second,
they can be seen as Wong--Zakai-type limits of continuous random ODEs.

Indeed, for a pure jump L\'evy process $Z$ we define the polygonal approximations $Z^h$ given by
\begin{equation}
 \label{e:polZ}
Z^h_t=Z_{kh}+\frac{t-kh}{h}\Big(Z_{(k+1)h}-Z_{kh}\Big),\quad t\in[kh,(k+1)h),\quad k\in\mathbb N_0.
\end{equation}
Let $c(\cdot)$ be a globally Lipschitz continuous function.
Then the random ODE $\dot X^h=a(X^h) + b(X^h)\dot W^h + c(X^h)\dot Z^h$ has a unique continuous
solution a.s.
How can one quantify the convergence $X^h\to X$ as $h\to 0$?

First we note, that since $Z$ has jumps, the approximations
$Z^h$ do not converge to $Z$ either uniformly or in the $J_1$-Skorokhod topology.
However, for each $t\in[0,\infty)$ we have $Z^h_t\to Z_t$ a.s. Alternatively, $Z^h$ converges to $Z$ a.s.\ in the (strong)
$M_1$-Skorokhod topology (see, e.g.,\ \cite{PavlyukevichR-15,HarPav-19}). It is clear, that the same restrictions apply to
the convergence of the solutions $X^h\to X$.

Thus, in Theorem 1 in \cite{Marcus-78},
S.\ Marcus proved that for the SDE driven by independent Poisson processes,
and polygonal approximations of the noise,
$X_t^h\to X_t$ a.s.\ for
almost all $t\in [0,T]$. Similar results in a more general setting were obtained in \cite{Marcus-78}, Theorems 2 and 3.

Kunita \cite{Kunita-95} showed in Corollary from Theorem 4 that the finite dimensional distributions of $X^h$ converge weakly
to the corresponding finite dimensional distributions of $X$.

Finally, Kurtz \emph{et al.} \cite{KurtzPP-95}
considered the mollified absolutely continuous approximations $Z^h_t:=h^{-1}\int_{(t-h)_+}^{t} Z_s\,\di s$ and showed that
$X_t^h\to X_t$ in probability for all but countably many $t\in[0,T]$, see Theorem 6.5 in \cite{KurtzPP-95}.

The present work is devoted to the thorough analysis of the strong convergence $X^h\to X$
which includes the study of uniform convergence with respect to the initial value and determination
of the convergence rate.
The main result of this paper is the improved analog of Theorem \ref{t:11} for solutions of Marcus SDEs.

In Section \ref{s:setting} we formulate the set of assumptions and
present the main results. The proofs of are given in Sections \ref{s:flows}--\ref{s:main proof}.

In this paper, we will use the following notation. The norm and scalar product in Euclidean spaces are denoted by $|\cdot|$ and
$\langle\cdot,\cdot\rangle$. For a matrix $A$, $\|A\|$ denotes its operator norm, and for a function $f=f(x)$, $\|f\|$ denotes
its supremum norm. For a function $f=f(x)$, the (partial) derivatives are denoted by
$\partial_x f=f_x$.
We write $f(x)\leq_C g(x)$ if there exists a constant $C\in (0,\infty)$ such that
$f(x)\leq C g(x)$ for all $x$, and we do not need this constant for further reference.
In the case that the functions $f$, $g$ depend on additional parameters (e.g., $t$, $h$, etc)
the above inequality should hold true uniformly over these parameters.

\section{Setting and the main result\label{s:setting}}

Let $(\Omega,\rF,\mathbb F,\P)$ be a filtered probability space satisfying the usual hypotheses. For $m\in\mathbb N$,
let $W$ be an $m$-dimensional
Brownian motion and $Z$ be an independent $m$-dimensional square integrable L\'evy process
with the characteristic function
$$
\E \ex^{\i \langle Z_t,\xi\rangle}=\exp\Big(t\int (\ex^{\i \langle z,\xi\rangle}-1-\i \langle z,\xi\rangle)\,\nu(\di z)\Big),
\quad \xi\in\bR^m,\  t\in[0,\infty),
$$
where the L\'evy measure $\nu$ on $(\bR^m\backslash{\{0\}},\rB(\bR^m\backslash{\{0\}}))$
is such that $\int_{|z|>0} |z|^2\,\nu(\di z)<\infty$.
By the L\'evy--It\^o decomposition, the L\'evy process $Z$ can be represented as an integral
$$
Z_t=\int_0^t \int_{|z|> 0} z\, \widetilde N(\di s,\di z)
$$
with respect to the compensated Poisson random measure $\widetilde N$ on $[0,\infty)\times (\bR^m\backslash{\{0\}})$
with the intensity measure $\nu(\di z)\,\di t$.
The reader may consult the \cite{Applebaum-09,Kunita-04,Sato-99} for more information on L\'evy processes
and stochastic calculus.

For $d\in\mathbb N$, let us consider a vector-valued function
$a\colon \mathbb{R}^d\to\mathbb{R}^d$, $a=(a^i(\cdot))_{i=1,\dots,d}$,
and matrix-valued functions $b\colon \mathbb{R}^d\to\mathbb{R}^{d\times m}$, $b=(b^i_j(\cdot))_{i=1,\dots,d, j=1,\dots, m}$, and
$c\colon\mathbb{R}^d\to\mathbb{R}^{d\times m}$, $c=(c^i_j(\cdot))_{i=1,\dots,d, j=1,\dots, m}$.
For the mapping $b$, $b^i$ denotes the $i$-th row of the matrix $b$, and $D b^i(x)$ denotes its gradient matrix.
By $\operatorname{tr}(Db(x)b(x))$ we mean a vector with coordinates $\operatorname{tr}(Db^i(x)b(x))$.
The (canonical) Marcus SDE under consideration is formally written in the following form:
\begin{equation}
\label{e:SDEM}
X_t=x+\int_0^t a(X_s)\, \di s+\int_0^t b(X_s)\circ\di W_s+\int_0^t c(X_s)\diamond \di Z_s,\quad x\in\bR^d,\ t\in[0,\infty).
\end{equation}
As already mentioned above, the integral $\int\circ\, \di W$ is understood in the Stratonovich sense.
The jumps of the process $X$ are determined as follows.
For each $x\in\bR^d$ and $z\in\bR^m$, consider the auxiliary ODE
\ba
\label{e:ophi-eq}
\frac{\di}{\di u} \varphi^z(u;x)&=c(\varphi^z(u;x))z,\quad u\in[0,1],\\
\varphi^z(0;x)&=x.
\ea
Assume that this ODE has a unique global solution and define
the Marcus flow
\ba
\label{e:op}
&\varphi^z(x):=\varphi^z(1;x).
\ea
Then the jump size of $X$ at time instant $\tau$ is defined as
$\Delta X_\tau=\varphi^{\Delta Z_\tau}(X_{\tau-})-X_{\tau-}$. Eventually, the Marcus SDE \eqref{e:SDEM} is
understood as the following It\^o SDE driven by the Brownian motion $W$ and the compensated Poisson random measure $\widetilde N$:
\ba
\label{e:SDEI}
X_t=x&+
\int_0^t a(X_{s})\,\di s
+    \int_0^t b(X_{s})\,\di W_s
+  \frac12 \int_0^t \operatorname{tr}(Db(X_{s})b(X_{s}))  \,\di s\\
&+ \int_0^t\int_{|z|>0} \Big(\varphi^z(X_{s-}) -X_{s-}  \Big) \, \widetilde N(\di z,\di s) \\
&+ \int_0^t\int_{|z|>0} \Big(\varphi^z(X_{s}) -X_{s} - c(X_{s})z \Big)\, \nu(\di z)\,\di s.
\ea
Let us formulate the conditions for the existence and uniqueness of solutions to \eqref{e:SDEM} and recall some of their properties.
To this end, we introduce the following set of assumptions on the coefficients $a$, $b$ and $c$ that will be used in the sequel.

\smallskip

\noindent
\textbf{H$_{a}^{(r)}$:}
For $r\in\mathbb N$, the following holds true:
\ban
&a\in C^r(\bR^d,\bR^d) \text{ and }  &&\|\partial^\alpha  a^i\|<\infty ,
\quad 1\leq i \leq d, \ 1\leq |\alpha|\leq r,\quad 1\leq |\alpha|\leq r.
\ean

\smallskip

\noindent
\textbf{H$_{b}^{(r)}$:}
For $r\in\mathbb N$, the following holds true:
\ban
&b\in C^r(\bR^d,\bR^{d\times m})   \text{ and }   &&\|\partial^\alpha b^i_j\|<\infty ,\quad  1\leq i\leq d, \ 1\leq j\leq m, \ 1\leq |\alpha|\leq r;\\
&&&\| b^i_j\cdot \partial^\alpha b^k_l\|<\infty, \quad 1\leq i,k\leq d, \ 1\leq j,l\leq m, \ 2\leq |\alpha|\leq r.
\ean

\smallskip

\noindent
\textbf{H$_{c}^{(r)}$:}
For $r\in\mathbb N$, the following holds true:
\ban
&c\in C^r(\bR^d,\bR^{d\times m})   \text{ and }    &&\|\partial^\alpha c^i_j\|<\infty ,\quad  1\leq i\leq d, \ 1\leq j\leq m,
\quad 1\leq |\alpha|\leq r;\\
&&&\| c^i_j\cdot \partial^\alpha c^k_l\|<\infty, \quad 1\leq i,k\leq d, \ 1\leq j,l\leq m,\ 2\leq |\alpha|\leq r.
\ean

\begin{rem}
For $r=1$,  the conditions including derivatives of the second and higher orders are void.
\end{rem}

For the mapping $c\colon \mathbb{R}^d\to\mathbb{R}^{d\times m}$, by $Dc(x)$ we denote its gradient tensor
$Dc(x)=\{\partial_{x^i} c^j_k(x)\}_{1\leq i,j\leq d, 1\leq k\leq m}$.
For each $x\in\bR^d$, we define
\ban
\|Dc(x)\|:=\sup_{\bR^m\ni|z|\leq 1}\|D (c(x)z)\|
\ean
and we set
\ban
\|Dc\|:=\sup_{x\in\mathbb{R}^d}\|Dc(x)\|.
\ean

Finally, we introduce the following assumption concerning the integrability of the L\'evy measure $\nu$.

\smallskip

\noindent
\textbf{H$_{\nu,A+}$:} For $A\in[0,\infty)$, there is $\e\in(0,\infty)$ such that the tails of the L\'evy measure $\nu$ are exponentially
integrable:
$$
\int_{|z|>1} \ex^{(A+\e)|z|}\nu(\di z)<\infty.
$$
\begin{rem}
Let $A\in[0,\infty)$. By Theorems 25.3 and 25.18 from Sato \cite{Sato-99}, condition \textbf{H$_{\nu,A+}$}
implies that for each $T\in[0,\infty)$ and
any $p\in[0,\infty)$
$$
\E \sup_{t\in[0, T] }|Z_t|^p\ex^{A|Z_t|}<\infty.
$$
If the support of $\nu$ is bounded, the condition \textbf{H$_{\nu,A+}$} is satisfied for any $A\in [0,\infty)$.
\end{rem}

\begin{thm}
\label{t:existenceSDE}
Assume that conditions \emph{\textbf{H$^{(1)}_{a}$}}, \emph{\textbf{H$^{(2)}_{b}$}}, \emph{\textbf{H$^{(2)}_{c}$}}
and \emph{\textbf{H$_{\nu,p\|Dc\|+}$}} hold true for some $p\in[2,\infty)$.
Then there is a unique strong solution $X=(X_t)_{t\geq 0}$ of the SDE \eqref{e:SDEI}. Moreover,
for any $T\in[0,\infty)$ there is a constant $C\in(0,\infty)$ such that the following holds true:
\ba
\label{e:sup x}
\E \sup_{0\leq t\leq T} |X_t(x) |^p & \leq C(1+ |x|^p),\quad x\in\bR^d,\\
\E \sup_{0\leq t\leq T} |X_t(x)-X_t(y) |^p & \leq C |x-y |^p,\quad x,y\in\bR^d.
\ea
\end{thm}

The numerical approximation of the solution $(X_t)_{t\geq 0}$ is based on the non-linear Wong--Zakai  scheme. For the time step
$h\in(0,1]$,
we approximate the trajectories of $W$ and $Z$ by polygonal lines defined in \eqref{e:polW} and \eqref{e:polZ} respectively,
and consider a random non-autonomous ODE
\ba
\label{SDE_M_h}
X_t^h&=x+\int_0^t \Big(a(X^h_s)+b(X^h_s) \dot W_s^h+c(X^h_s)\dot Z_s^h\Big)\, \di s,\quad t\in[0,\infty).
\ea
The solution of the ODE \eqref{SDE_M_h} exists and is
unique a.s.\ under assumptions of global Lipschitz continuity of the coefficients $a$, $b$ and $c$, and is obtained
by gluing together the solutions of random autonomous ODEs
on each interval $[kh,(k+1)h]$, $k\in \mathbb N_0$.

As it was mentioned above, since the trajectories $t\mapsto X^h_t$ are continuous,
one cannot expect convergence of $X^h$ to $X$ in the uniform or
in the Skorokhod $J_1$ topologies. However one can ask if there is uniform convergence of
the approximations evaluated at the knots $\{kh\colon 0\leq kh\leq T\}$.

To this end, let us obtain the values $X^h_{kh}$ as follows.
For $\tau\in[0,\infty)$ and $w,z\in\bR^m$, consider the ODE
\ba
\label{e:eqpsi}
\frac{\di}{\di u} \psi(u)&=a(\psi(u))\tau+   b(\psi(u))w+c(\psi(u))z,\\
\psi(0)&=x,\quad u\in[0,1],
\ea
where $\psi(u)=\psi(u;x)=\psi(u;x;\tau,w,z)$.
Let
\ba
\label{e:psi}
\Psi(x; \tau,w,z):=\psi(1;x,\tau ,w,z).
\ea
Then the values $X^h_{kh}$ are obtained recursively as a time-discrete Wong--Zakai  numerical scheme:
\ba
\label{e:Euler-d}
X^h_0&=x,\\
X^h_{(k+1)h}&=\Psi(X^h_{kh}; h, W_{(k+1)h}-W_{kh},Z_{(k+1)h}-Z_{kh}),\quad k\in\mathbb N_0.
\ea

We also consider the time-continuous c\`adl\`ag extension of this scheme given by the process
\ba
\label{e:Euler}
\bar X^h_0&=x,\\
\bar X^h_t&=\Psi(\bar X^h_{kh}; t-kh, W_t-W_{k h},Z_t-Z_{k h}),\quad t\in (kh,(k+1)h],\quad k\in\mathbb N_0.
\ea
By construction, $\bar{X}^h_{kh}=X^h_{kh}$ for all $k\in \mathbb N_0$.
Moreover the trajectories $t\mapsto\bar X^h_t$ are c\`adl\`ag and the jump times of $Z$, $X$ and $\bar X^h$ coincide.

To prove strong convergence of the approximations $X^h$ to $X$, we impose certain growth and regularity
assumptions on the function $\Psi$ defined in \eqref{e:psi}.

\smallskip
\noindent
\textbf{H$_{\Psi,K}$:}
There exist $C\in(0,\infty)$ and $\kappa_1,\kappa_2,K\in[0,\infty)$
such that for all $x,y\in\mathbb{R}^d,\tau\in\mathbb{R}_+,w,z\in\mathbb{R}^m$
\ban
\left.
\begin{aligned}
&| \Psi(x;\tau,w,z)|, \ |\partial_{\iota_1} \Psi(x;\tau,w,z)|,\\
&|\partial_{\iota_1\iota_2} \Psi(x;\tau,w,z)|,\  |\partial_{\iota_1\iota_2\iota_3} \Psi(x;\tau,w,z)|\\
&\iota_1 , \iota_2,\iota_3\in\{\tau,w^1,\dots,w^m,z^1,\dots,z^m\}\\
\end{aligned}
\right\}
&\leq C(1+|x|)\ex^{\kappa_1\tau + \kappa_2|w| + K |z|}\\
\left.
\begin{aligned}
&|\Psi(x;\tau,w,z)- \Psi(y;\tau,w,z)|,\\
&|\partial_\tau \Psi(x;\tau,w,z)-\partial_\tau \Psi(y;\tau,w,z)|\\
&|\partial_{\tau \iota_1} \Psi(x;\tau,w,z)-\partial_{\tau \iota_1} \Psi(y;\tau,w,z)|\\
&|\partial_{\iota_1\iota_2} \Psi(x;\tau,w,z)-\partial_{\iota_1\iota_2} \Psi(y;\tau,w,z)|\\
&\iota_1\in\{w^1,\dots,w^m\},\  \iota_2\in\{z^1,\dots,z^m\}\\
\end{aligned}
\right\}
&\leq C|x-y|\ex^{\kappa_1\tau + \kappa_2|w| + K |z|}.
\ean
\begin{rem}
In the assumption \textbf{H$_{\Psi,K}$} we single out the constant $K\in[0,\infty)$
that will be related with the exponential integrability of the L\'evy process
$Z$, i.e., the existence of the expectations $\E\sup_{t\leq T} |Z_t|^p \ex^{K|Z_t|}$ for $p\in[0,\infty)$.
The constants $C$, $\kappa_1$ and $\kappa_2$ make no influence either on moment estimates of the numerical approximation
$\bar X^h$ or on the convergence rate due to the fact that $\E \sup_{t\leq T}|W_t|^p\ex^{\kappa_1 t +\kappa_2|W_t|}<\infty$ always.
\end{rem}

\begin{lem}
\label{l:bounded}
Let assumptions \emph{\textbf{H$_{a}^{(3)}$}}, \emph{\textbf{H$_{b}^{(3)}$}}
and \emph{\textbf{H$_{c}^{(3)}$}} hold true, and let additionally all the functions $a$, $b$ and $c$ be bounded.
Then condition \emph{\textbf{H$_{\Psi,K}$}} holds true for any
\ba
\kappa_1>5\|Da\|,\quad \kappa_2>5\|Db\|,\quad K>5\|Dc\|.
\ea
\end{lem}

The approximations $\bar X^h$ have the following properties.
\begin{thm}
\label{t:existenceWZ}
Assume that conditions \emph{\textbf{H$_{a}^{(3)}$}}, \emph{\textbf{H$_{b}^{(3)}$}}
\emph{\textbf{H$_{c}^{(3)}$}}, \emph{\textbf{H$_{\Psi,K}$}} and \emph{\textbf{H$_{\nu,pK+}$}}
hold true for some $K\in[0,\infty)$ and $p\in[2,\infty)$.
Then for any $T\in[0,\infty)$ there is a constant $C\in(0,\infty)$ such that for
any $x,y\in\bR^d$ and any $h\in(0,1]$
the approximation scheme $\{\bar X^h_t\}_{0\leq t\leq T}$ satisfies
\begin{align}
\label{e:sup x bar}
\E\sup_{0\leq t\leq T} |\bar{X}^h_t(x) |^p &\leq C(1+|x|^p),\\
\label{e:supXxy}
\E  \sup_{0\leq t\leq T} |\bar{X}^h_t(x)-\bar{X}^h_t(y) |^p &\leq C|x-y|^p.
\end{align}
\end{thm}
The next Theorem contains the first main result on the strong approximation of solutions $X$ of the Marcus SDE \eqref{e:SDEM}.
\begin{thm}
\label{thm:Xx-Xhx}
Assume that assumptions \emph{\textbf{H$_{a}^{(3)}$}}, \emph{\textbf{H$_{b}^{(3)}$}},
\emph{\textbf{H$_{c}^{(3)}$}}, \emph{\textbf{H$_{\Psi,K}$}}, \emph{\textbf{H$_{\nu,pK+}$}}
and \emph{\textbf{H$_{\nu,p\|Dc\|+}$}}
hold true for some $K\in[0,\infty)$ and $p\in[2,\infty)$.
Then for any $T\in[0,\infty)$ there is a constant $C\in(0,\infty)$ such that for any
$h\in(0,1]$
$$
\E\sup_{0\leq t\leq T}|X_t(x)-\bar{X}^h_t(x)|^p\leq Ch(1+|x|^p),\quad x\in\bR^d.
$$
\end{thm}
The second main result of this paper is the
strong locally uniform approximation by the Wong--Zakai scheme \eqref{e:Euler} w.r.t.\ the initial point $x\in\bR^d$.
\begin{thm}
\label{t:main}
Assume that assumptions \emph{\textbf{H$_{a}^{(3)}$}}, \emph{\textbf{H$_{b}^{(3)}$}},
\emph{\textbf{H$_{c}^{(3)}$}}, \emph{\textbf{H$_{\Psi,K}$}}, \emph{\textbf{H$_{\nu,2d\|Dc\|+}$}} and \emph{\textbf{H$_{\nu,2dK+}$}}
hold true.
Then for any $T\in[0,\infty)$, $N\in(0,\infty)$ and $\e\in (0,1)$ there is $C\in(0,\infty)$ such that
for any $h\in(0,1]$
$$
\E \sup_{|x|\leq N}  \sup_{0\leq t\leq T}    |X_t(x)-\bar{X}^h_t(x)|\leq C h^{\frac{1-\e}{4d}}.
$$
\end{thm}

\smallskip

As we see, our results give the convergence rate of order $\frac12$,
provided that the coefficients are $C^3$.
Clearly, all the results formulated above hold true for solutions $X^\circ$
of continuous Stratonovich SDEs with $c\equiv 0$. Hence,
Theorem \ref{thm:Xx-Xhx}
complements the findings of Theorems \ref{t:11}--\ref{t:14}.
The estimates for the uniform convergence given in
Theorem \ref{t:main} is novel and refine the convergence of Theorem \ref{t:11}.

Eventually we note that the numerical scheme $X^h$ has the weak convergence of order 1, as it was shown by
Kosenkova \emph{et al.} in \cite{KoKuPa-19}.
\begin{thm}[Theorem 2.4 in \cite{KoKuPa-19}]
 Let assumptions \emph{\textbf{H$_{a}^{(4)}$}} and \emph{\textbf{H$_{c}^{(4)}$}}, \emph{\textbf{H$_{c}^{(4)}$}}
 and \emph{\textbf{H$_{\nu,8\|Dc\|+}$}}  hold true. Then for any $f\in C_b^4(\bR^d,\bR)$ and
any $T\in[0,\infty)$ there is a constant $C\in (0,\infty)$ such that that
for any $h\in(0,1]$
$$
\max_{k\in\mathbb N_0\colon kh\leq T} |\E f(X^h_{kh}) - \E f(X_{kh})|\leq C h(1+|x|^4),\quad  x\in\mathbb R^d .
$$
 \end{thm}

\section{Properties of the mappings $\phi^z(\cdot)$  and $\Psi(\cdot)$\label{s:flows}}

The properties of the Marcus flow
$\phi^z(\cdot)$ defined in \eqref{e:ophi-eq} and \eqref{e:op} are decisive for our analysis of convergence.

The next Lemma
contains elementary estimates of the Marcus mapping $x\mapsto \phi^z(x)$. Similar estimates can be found, e.g.,
in Lemma 3.1 in Fujiwara and Kunita \cite{FujiwaraK-99a}. They follow immediately from the Gronwall Lemma; we omit the proof.
\begin{lem}
\label{t:phi}
Let the coefficient $c$ satisfy the conditions $\emph{\textbf{H}}^{(2)}_{c}$. Then
for each $x\in\bR^d$ and $z\in\bR^m$, the following estimates hold true:
\begin{subequations}
\begin{align*}
(1)\qquad &\sup_{u\in[0,1]}|\phi^z(u;x)|
\leq_C (|x| + |z|) \ex^{\|Dc\| |z|},
\\
(2)\qquad & |\phi^z(x)-x |\leq_C  
|z|\ex^{ \|Dc\||z|}(1+|x|) ,\\
(3)\qquad & |\phi^z(x)-\phi^z(y) |\leq    \ex^{ \|Dc\||z|}|x-y| ,\\
(4)\qquad & |\phi^z(x)-x-  \phi^z(y)+y |\leq_C 
|z|\ex^{ \|Dc\||z|}|x-y| ,\\
(5)\qquad & |\phi^z(x)-x -c (x) z |\leq_C 
|z|^2\ex^{\|Dc\| |z|}(1+|x|),  \\
(6)\qquad & |\phi^z(x)-x -  c (x) z- \phi^z(y)+y + c (y) z |\leq_C 
|z|^2 \ex^{\|Dc\| |z|} |x-y| .
\end{align*}
\end{subequations}
\end{lem}

Analogously, the application of the Gronwall Lemma to the equation \eqref{e:eqpsi} yields the following estimates
(compare with (1), (2), and (3) in Lemma \ref{t:phi}).
\begin{lem}
Let the conditions $\emph{\textbf{H}}^{(1)}_{a}$, $\emph{\textbf{H}}^{(1)}_{b}$ and $\emph{\textbf{H}}^{(1)}_{c}$. Then
for any $x,y\in\mathbb{R}^d$, $\tau\in\mathbb{R}_+$, $w,z\in\mathbb{R}^m$ we have
\label{l:psi}
\ban
(1)\quad &|\Psi(x;\tau,w,z)| \leq_C(1+|x|)              \ex^{\|Da\|\tau+\|Db\| |w|+\|Dc\| |z|},\\
 (2)\quad &|\Psi(x;\tau,w,z)-x|\leq_C(1+|x|)( \tau +  |w|+|z|)\ex^{\|Da\|\tau + \|Db\| |w|+\|Dc\| |z|},\\
(3)\quad &|\Psi(x;\tau,w,z)-\Psi(y;\tau,w,z)| \leq_C|x-y|\ex^{\|Da\|\tau+\|Db\| |w|+\|Dc\||z|}.
\ean
\end{lem}

\begin{lem}
\label{l:cond}
Let $p\in[2,\infty)$ be fixed and let the condition $\mathbf{H}_{\nu,pK+}$ be satisfied.
Then there is $C\in(0,\infty)$ such that for $h\in(0,1]$
\begin{equation*}
\sup_{0\leq s\leq h}\E (s+|W_s|+|Z_s|)^p \ex^{p\kappa_1 s+p\kappa_2|W_s|+pK|Z_s|}\leq Ch.
\end{equation*}
\end{lem}
\begin{proof}
Let $0\leq s\leq h\leq 1$.
Due to the independence of $W$ and $Z$,
it is sufficient to estimate the expectations
\ban
\E |W_s|^p \ex^{pK|W_s|} \quad\text{ and }\quad \E |Z_s|^p \ex^{pK|Z_s|} .
\ean
The estimate for $W$ follows from the self-similarity of the Brownian motion.
To get the estimate for $Z$, we assume for simplicity that $m=1$.
We take into account that $\ex^{|x|}\leq \ex^{x}+ \ex^{-x}$, so that it is sufficient to estimate
$\E |Z_s|^p \ex^{pKZ_s}$.

Let $g(z):=z\ex^{Kz}$. By the It\^o formula, we have
\ban
g(Z_t)
&= \int_0^t\int_{|z|>0} ( g(Z_{s-}+z)- g(Z_{s-})) \,\widetilde N(\di z,\di s) \\
&+
\int_0^t\int_{|z|>0}  ( g(Z_{s}+z) - g(Z_{s}) - ( 1+KZ_{s}) \ex^{KZ_{s-}} z)    \,\nu(\di z)\,\di s,
\ean
and by Novikov's inequality (Theorem 4.20 in \cite{kuhn2023maximal}) we get
\ban
\E\sup_{s\leq h}|Z_s|^p\ex^{p K Z_s}&
\leq_C \E\int_0^h \Big( \int_{|z|>0}
| g(Z_{s}+z) - g(Z_{s})|^2 \,\nu(\di z)\Big)^{p/2}\,\di s\\
&+\E \int_0^h \int_{|z|>0} | g(Z_{s}+z) - g(Z_{s}) |^p \,\nu(\di z)\,\di s\\
&+\E \Big(\int_0^h \int_{|z|>0}
 | g(Z_{s}+z) - g(Z_{s})-   ( 1+KZ_{s}  ) \ex^{KZ_{s}} z  |
\,\nu(\di z)\,\di s\Big)^p\\
&\leq_C h
\ean
under the condition
$$
\int_{|z|>1} |z|^{p}\ex^{pK|z|}\,\nu(\di z) \leq C \int_{|z|>1} \ex^{(pK+\e)|z|}\,\nu(\di z)<\infty.
$$
\end{proof}

\section{Proof of Lemma \ref{l:bounded}}

To simplify the notation, we consider the one-di\-men\-si\-o\-nal case $d=m=1$ only. We will still use the notation $Da=a'$ etc.
The multivalued estimates are straightforward.
We estimate the derivatives of $\Psi$. The Lipschitz property of $\Psi$ and its derivatives is studied analogously.

\noindent
\textbf{0. Estimate of $\Psi$.}
Since the coefficients $a$, $b$ and $c$ are bounded, we get
\ban
\sup_{u\in[0,1]}|\psi(u;x;\tau,w,z)|\leq |x| + \|a\|\tau + \|b\||w|+\|c\||z| \leq_C (1+|x|)\ex^{\e( \tau+ |w|+|z|)}
\ean
for any $\e>0$.

\noindent
\textbf{1. Estimates of $\Psi_\tau$, $\Psi_{w^i}$, $\Psi_{z^j}$.}
The derivative w.r.t.\ $\tau$ satisfies the linear non-autonomous ODE
\ban
\frac{\di}{\di u}&\psi_\tau=a(\psi)+(Da (\psi)\tau +Db(\psi)w+Dc(\psi)z) \psi_\tau,\quad
\psi_\tau(0;x;\tau,w,z)=0.
\ean
Hence the Gronwall Lemma yields
\ban
\sup_{u\in[0,1]}|\psi_\tau(u)|&=\|a\| \ex^{\|Da\|\tau +\|Db\||w|+\|Dc\||w|}\leq\|a\|(1+|x|)\ex^{\|Da\|\tau +\|Db\||w|+\|Dc\||w| }.
\ean
A similar estimate holds for all first order derivatives.

\noindent
\textbf{2. Estimates of  $\Psi_{\tau\tau}$, $\Psi_{\tau w^j}$, $\Psi_{\tau z^j}$, $\Psi_{w^jw^k}$, $\Psi_{w^jz^k}$, $\Psi_{z^jz^k}$.}
We consider the derivatives $\Psi_{\tau\tau}$ and $\Psi_{\tau w}$ only. We have
\ban
\frac{\di}{\di u}& \psi_{\tau\tau}
=2 Da (\psi) \psi_\tau  + \Big(D^2a(\psi)\tau  +D^2b(\psi)w +D^2c(\psi)z\Big)\psi_\tau^2\\
&+ \Big(Da(\psi)\tau  + Db(\psi)w + Dc(\psi)z\Big) \psi_{\tau\tau},\\
&\psi_{\tau\tau}(0;x;\tau,w,z)=0,\\
\frac{\di}{\di u} &\psi_{\tau w}=Da(\psi)\psi_w+ Db(\psi)\psi_\tau+
\Big(D^2 a(\psi)\tau +D^2 b(\psi)w + D^2 c(\psi)z\Big) \psi_\tau  \psi_w\\
&\qquad +\Big(Da(\psi)\tau +Db(\psi)w+Dc(\psi)z\Big) \psi_{\tau w},\\
&\psi_{\tau w}(0;x;\tau,w,z)=0.
\ean
Applying the Gronwall Lemma and using the estimates from the previous steps yields
\ban
|\psi_{\tau\tau}(u)|&\leq_C (\tau+|w|+|z|) \ex^{3\|Da\|\tau +3\|Db\||w|+3\|Dc\||w|}\\
&\leq_C(\tau+|w|+|z|) (1+|x|)\ex^{3\|Da\|\tau +3\|Db\||w|+3\|Dc\||w| },\\
|\psi_{\tau w}(u)|&\leq_C (\tau+|w|+|z|) \ex^{3\|Da\|\tau +3\|Db\||w|+3\|Dc\||w|}\\
&\leq_C (\tau+|w|+|z|) (1+|x|)\ex^{3\|Da\|\tau +3\|Db\||w|+3\|Dc\||w| }.
\ean

\noindent
\textbf{3. Estimates of $\Psi_{\tau\tau\tau}$, $\Psi_{\tau \tau w^j}$, $\Psi_{\tau \tau  z^j}$, $\Psi_{\tau w^jw^k}$, \dots}
We consider the derivative $\Psi_{\tau\tau\tau}$:
\ban
&\frac{\di}{\di u}\psi_{\tau\tau\tau}
=3D^2 a(\psi) \psi_\tau^2 + 3Da(\psi) \psi_{\tau\tau} + 3\Big(D^2 a(\psi)\tau  +D^2 b(\psi)w +D^2 c(\psi)z\Big)\psi_\tau\psi_{\tau\tau}\\
&\qquad + \Big(D^3 a(\psi)\tau     +D^3 b(\psi)w  +D^3 c(\psi)z \Big)\psi_\tau^3
+ \Big(D a(\psi)\tau  + D b(\psi)w + D c(\psi)z\Big) \psi_{\tau\tau\tau},\\
&\psi_{\tau\tau\tau}(0;x;\tau,w,z)=0.
\ean
Hence we get
\ban
|\psi_{\tau\tau\tau}(u)|\leq C_2 (\tau+|w|+|z|)^2 (1+|x|)\ex^{5\|Da\|\tau +5\|Db\||w|+5\|Dc\||w| }.
\ean
A similar estimate holds for all third order derivatives.

\section{Proof of Theorem \ref{t:existenceSDE}\label{s:existence}}

The statement of Theorem \ref{t:existenceSDE} follows immediately from Theorems 3.1 and 3.2 in Kunita \cite{Kunita-04}.
Indeed, we have to check the following conditions.

The functions
\ban
x&\mapsto a^i(x),\quad i=1,\dots, d,\\
x&\mapsto b^i(x),\quad i=1,\dots,d,\\
x&\mapsto \frac{\partial}{\partial x^l}b^i(x) b_j^l(x) , \quad i,l=1,\dots, d,\ j=1,\dots, m,
\ean
are globally Lipschitz continuous and of linear growth by assumptions $\mathbf{H}_{a}^{(1)}$,
$\mathbf{H}_{b}^{(2)}$ and
$\mathbf{H}_{c}^{(2)}$.

Due to Lemma \ref{t:phi} (5) and (6), the mapping
$$
x\mapsto \int_{|z|>0} \Big(\phi^z(x)-x - c(x)z \Big)\,\nu(\di z)
$$
is globally Lipschitz continuous under the condition $\mathbf{H}_{c}^{(2)}$ provided
\ban
\int_{|z|>0} |z|^2 \ex^{\|Dc\||z|}\,\nu(\di z)<\infty.
\ean
Due to Lemma \ref{t:phi} (2) and (4), for all $x,y\in\bR^d$ and $z\in\bR^m$ we have
\ban
\frac{|\phi^z(x)-x|}{1+|x|}  &\leq_C |z|\ex^{\|Dc\||z|},\\
|\phi^z(x)-x- \phi^z(y)+y  | &\leq_C |z|\ex^{\|Dc\||z|}|x-y|,\\
\ean
Thus for $p\in[2,\infty)$, condition (3.2) in Kunita \cite{Kunita-04} follows from the assumption  $\mathbf{H}_{\nu,p\|Dc\|+}$:
\ban
\int_{|z|>0} |z|^p \ex^{ p \|Dc\| |z|}\,\nu(\di z)
\leq_C \int_{0<|z|\leq 1} |z|^2 \,\nu(\di z)
+ \int_{|z|>1} \ex^{ (p \|Dc\|+\e) |z|}\,\nu(\di z)
<\infty.
\ean

\section{Proof of Theorem \ref{t:existenceWZ}}

To simplify the notation, we assume that $d=1$.
Let $p\in[2,\infty)$ and $t\in[0,T]$.
For $x\in\bR^d$ fixed, for brevity we will denote $X_t=X_t(x)$, $X^h_t=X^h_t(x)$ and $\bar{X}^h_t=\bar{X}^h_t(x)$.
It is easy to see that for each $x\in\bR$, the mapping
$$
(x;\tau,w,z)\mapsto \Psi(x;\tau,w,z)
$$
is $C^2$-smooth.

For the time step $h\in(0,1]$ and $t\in[0,\infty)$, we denote
\ba
n_t^h&:=\begin{cases}
         0,\quad t=0,\\
         k,\quad t\in(kh,(k+1)h],\quad k\in\mathbb N_0,
        \end{cases}\\
t^h &:=t-n_t^h h,\\
\Delta W^h_t& :=W_t -  W_{n_t^h h },\\
\Delta Z^h_t& :=Z_t -  Z_{n_t^h h }.
\ea
The continuous time Wong--Zakai type scheme $\bar X^h$, see \eqref{e:Euler}, is written in the following unified form:
\ba
\label{e:Euler1}
\bar X^h_0&=x,\\
\bar X^h_t&=\Psi(X^h_{n_t^h h}; t^h, \Delta W_t^h,\Delta Z_t^h),\quad t\in(0,\infty).
\ea
We also recall that $X^h_{kh}=\bar X^h_{kh}$, so that
$X^h_{n_t^hh}=\bar X^h_{n_t^hh}$ for any $t\in[0,\infty)$. Applying the It\^o formula on each interval
$t\in(kh,(k+1)h]$ we see that
the process $\bar{X}^h$ satisfies the following SDE:
\ba
\label{e:SDEbarX}
&\bar{X}^h_t= x+ \int_0^t \Psi_\tau(\bar X^h_{n_s^hh};s^h,\Delta W^h_s,\Delta Z^h_s)\,\di s \\
&+ \int_0^t \langle\Psi_w(\bar X_{n_s^hh};s^h,\Delta W^h_s,\Delta Z^h_s)\,\di  W_s\rangle
+\frac{1}{2} \int_0^t \operatorname{tr}\Big(\Psi_{ww}(\bar X^h_{n_s^hh};s^h,\Delta W^h_s,\Delta Z^h_s)\Big)\,\di s \\
&+\int_0^t\int_{|z|>0}\Big(\Psi(\bar X^h_{n_s^hh};s^h,\Delta W^h_s,\Delta Z^h_{s-}+z)-\Psi(\bar X^h_{n_s^hh};s^h,\Delta W^h_s,\Delta Z^h_{s-})\Big)\,
\widetilde{N}(\di z,\di s) \\
&+
\int_0^t\int_{|z|>0}\Big(\Psi(\bar X^h_{n_s^hh};s^h,\Delta W^h_s,\Delta Z^h_{s-}+z)-\Psi(\bar X^h_{n_s^hh};s^h,\Delta W^h_s,\Delta Z^h_{s-})\\
&\qquad \qquad\qquad- \langle \Psi_z(\bar X^h_{n_s^hh};s^h,\Delta W^h_s,\Delta Z^h_{s-}+z),z \rangle \Big)\, \nu(\di z)\,\di s .
\ea
Note that $\bar X^h_{n_s^hh}=\bar X^h_{n_{s-}^hh}=\bar X^h_{kh}$ for $s\in(kh,(k+1)h]$. Since
$\P(\Delta Z_{kh}=0,\ k\in \mathbb N_0)=1$
for each $h\in(0,1]$, for convenience we write $\bar X^h_{n_s^hh}$ instead of $\bar X^h_{n_{s-}^hh}$.
The Novikov maximal inequality (see, Eq.\ (4.17) in \cite{kuhn2023maximal}) yields:
\ba
\label{e:ku}
\E \sup_{s\in[0,t]}& |\bar{X}^h_s|^p\leq_C |x|^p
+  \int_0^t \E\Big|\Psi_\tau(\bar X^h_{n_s^hh};s^h,\Delta W^h_s,\Delta Z^h_s)\Big|^p\,\di s \\
& +\int_0^t\E \Big|\Psi_w(\bar X_{n_s^hh};s^h,\Delta W^h_s,\Delta Z^h_s)\Big|^p\,\di s
+ \int_0^t\E \Big|\operatorname{tr}\Big(\Psi_{ww}(\bar X^h_{n_s^hh};s^h,\Delta W^h_s,\Delta Z^h_s)\Big)\Big|^p\,\di s \\
&+\int_0^t\E \Big(\int_{|z|>0}\Big|\Psi(\bar X^h_{n_s^hh};s^h,\Delta W^h_s,\Delta Z^h_{s}+z)-\Psi(\bar X^h_{n_s^hh};s^h,\Delta W^h_s,\Delta Z^h_{s})\Big|^2\,
\nu(\di z)\Big)^\frac{p}{2}\,\di s \\
&+\int_0^t\E \int_{|z|>0}\Big|\Psi(\bar X^h_{n_s^hh};s^h,\Delta W^h_s,\Delta Z^h_{s}+z)-\Psi(\bar X^h_{n_s^hh};s^h,\Delta W^h_s,\Delta Z^h_{s})\Big|^p\,
\nu(\di z)\,\di s \\
&+
\int_0^t\E\Big|\int_{|z|>0}\Big(\Psi(\bar X^h_{n_s^hh};s^h,\Delta W^h_s,\Delta Z^h_{s-}+z)-\Psi(\bar X^h_{n_s^hh};s^h,\Delta W^h_s,\Delta Z^h_{s-})\\
&\qquad\qquad\qquad   -  \langle\Psi_z(\bar X^h_{n_s^hh};s^h,\Delta W^h_s,\Delta Z^h_{s-}+z),z \rangle\Big)\, \nu(\di z)\Big|^p\,\di s .
\ea
By assumption $\mathbf{H}_{\Psi,K}$ and independence of $\bar X^h_{n_s^hh}$ of $(\Delta W^h_s,\Delta Z^h_s)$ we have
\ba
\label{e:kest}
\E|\Psi_\tau(\bar X^h_{n_s^hh};s^h,\Delta W^h_s,\Delta Z^h_s)|^p
&\leq_C \E(1+|\bar X^h_{n_s^hh}|^p) \E \sup_{s\in [0,h]}\ex^{p\kappa_1 h+p\kappa_2| W_s|+pK|Z_s|}\\
&\leq_C (1+\E\sup_{r\in[0,s]} |\bar X^h_{r}|^p),\\
\ea
and analogously
\ban
\E|\Psi_w(\bar X^h_{n_s^hh};s^h,\Delta W^h_s,\Delta Z^h_s)|^p &\leq_C (1+\E\sup_{r\in[0,s]} |\bar X^h_{r}|^p),\\
\E|\operatorname{tr}(\Psi_{ww}(\bar X^h_{n_s^hh};s^h,\Delta W^h_s,\Delta Z^h_s))|^p &\leq_C(1+\E\sup_{r\in[0,s]} |\bar X^h_{r}|^p),\\
\ean
Since for $p\in[2,\infty)$ we can estimate
\ban
\Big|\Psi(x;\tau,w,z+\zeta)-\Psi(x;\tau,w,z)\Big|^p
&=\Big| \int_0^1 \langle \Psi_{z}(x;\tau,w,z+\theta\zeta),\zeta\rangle \,\di \theta\Big|^p\\
&\leq_C|\zeta|^p (1+|x|^p) \ex^{p\kappa_1 \tau +p\kappa_2|w| + pK|z|+pK|\zeta|}
\ean
and
\ban
\Big|\Psi(x;\tau,w,z+\zeta)-\Psi(x;\tau,w,z)
&-  \langle\Psi_{z}(x;\tau,w,z),\zeta \rangle\Big|\\
&\leq_C|\zeta|^2 (1+|x|) \ex^{\kappa_1 \tau +\kappa_2|w| + K|z|+K|\zeta|},
\ean
the integrals w.r.t.\ $\nu$ in \eqref{e:ku} have the same bound as in \eqref{e:kest}.
Therefore get the inequality
\ban
\E \sup_{s\in[0,t]} |\bar{X}^h_s|^p&\leq_C |x|^p+ \int_0^t \Big(1+\E \sup_{r\in[0,s]} |\bar{X}^h_r|^p\Big)\,\di s,\quad t\in[0,T],
\ean
and the hence by Gronwall's lemma  we get
\ban
\E \sup_{s\in[0,T]} |\bar{X}^h_s|^p\leq_C 1+|x|^p.
\ean
The estimate \eqref{e:supXxy} is obtained analogously with the help of assumption \textbf{H$_{\Psi,K}$} for the differences.

\section{Proof of Theorem \ref{thm:Xx-Xhx}}

To simplify the notation, we assume that $d=1$. Let $p\in[2,\infty)$.
For $x\in\bR^d$ fixed, for brevity we will denote $X_t=X_t(x)$, $X^h_t=X^h_t(x)$ and $\bar{X}^h_t=\bar{X}^h_t(x)$.

Firstly, we write the difference $X-\bar{X}^h$ in the It\^o integral form:
\ba
\label{e:diff}
&X_t-\bar{X}^h_t\\
&= \int_0^t a(X_s)\,\di s - \int_0^t \Psi_\tau(X^h_{n_s^hh},s^h,\Delta W^h_s,\Delta Z^h_s)\,\di s \\
&+ \int_0^t \langle b(X_s),\,\di W_s \rangle
- \int_0^t \langle \Psi_w(X_{n_s^hh},s^h,\Delta W^h_s,\Delta Z^h_s),\,\di  W_s\rangle \\
&+\frac{1}{2} \int_0^t \langle b'(X_s),b(X_s)\rangle \, \di s
-\frac{1}{2} \int_0^t\operatorname{tr}(\Psi_{ww}(X^h_{n_s^hh},s^h,\Delta W^h_s,\Delta Z^h_s))\,\di s \\
&+ \int_0^t\int_{|z|>0}\Big(\varphi^z(X_{s-})-X_{s-}\Big)\, \widetilde{N}(\di z,\di s) \\
&-\int_0^t\int_{|z|>0}\Big(\Psi(X^h_{n_s^hh},s^h,\Delta W^h_s,\Delta Z^h_{s-}+z)-\Psi(X^h_{n_s^hh},s^h,\Delta W^h_s,\Delta Z^h_{s-})\Big)\,
\widetilde{N}(\di z,\di s) \\
&+\int_0^t\int_{|z|>0}\Big(\varphi^z(X_{s})-X_s- \langle c(X_s),z\rangle\Big)\,\nu(\di z)\, \di s \\
&-
\int_0^t\int_{|z|>0}\Big(\Psi(X^h_{n_s^hh},s^h,\Delta W^h_s,\Delta Z^h_{s}+z)-\Psi(X^h_{n_s^hh},s^h,\Delta W^h_s,\Delta Z^h_{s})\\
&\qquad \qquad \qquad\qquad \qquad \qquad
- \langle\Psi_z(X^h_{n_s^hh},s^h,\Delta W^h_s,\Delta Z^h_{s}+z),z\rangle \Big)\, \nu(\di z)\,\di s.
\ea
Let us introduce the following processes that appear in \eqref{e:diff}
\ban
F^a(s):=& a(X_s) - \Psi_\tau (X^h_{n_s^hh},s^h,\Delta W^h_s,\Delta Z^h_s),\\
F^{b}(s):=& b(X_s) - \Psi_w(X^h_{n_s^hh},s^h,\Delta W^h_s,\Delta Z^h_s),\\
F^{b'b}(s):=& \langle b'(X_s),b(X_s) \rangle
- \operatorname{tr}(\Psi_{ww}(X^h_{n_s^hh},s^h,\Delta W^h_s,\Delta Z^h_s)),\\
F^{\tilde{N}}(s-,z):=& \varphi^z(X_{s-}) - X_{s-} \\
&- \Psi(X^h_{n_s^hh},s^h,\Delta W^h_s,\Delta Z^h_{s-}+z)+\Psi(X^h_{n_s^hh},s^h,\Delta W^h_s,\Delta Z^h_{s-}) , \\
F^{\nu}(s,z):=& \varphi^z(X_{s})-X_s- \langle c (X_s),z \rangle \\
&- \Psi(X^h_{n_s^hh},s^h,\Delta W^h_s,\Delta Z^h_{s}+z)-\Psi(X^h_{n_s^hh},s^h,\Delta W^h_s,\Delta Z^h_s) \\
&\qquad \qquad \qquad + \Psi_z(X^h_{n_s^hh},s^h,\Delta W^h_s,\Delta Z^h_s+z),z\rangle.
\ean
We decompose each of these terms into a sum of three terms as follows:
\ba
\label{e:F}
F^a_1(s)& =   a(X_s)-a(\bar{X}^h_s), \quad F^a_2(s) =  a(\bar{X}^h_s)-a(X^h_{n_s^hh}), \\
F^a_3(s)& =  a(X^h_{n_s^hh})-\Psi_\tau(X^h_{n_s^hh},s^h,\Delta W^h_s,\Delta Z^h_s), \\
F^{b}_1(s)& =b(X_s)-b(\bar{X}^h_s), \quad F^{b}_2(s) =b(\bar{X}^h_s)-b(X^h_{n_s^hh}) ,\\
F^{b}_3(s)& =b(X^h_{n_s^hh})-\Psi_w(X^h_{n_s^hh},s^h,\Delta W^h_s,\Delta Z^h_s), \\
F^{b'b}_1(s)& =\langle b'(X_s), b(X_s) \rangle - \langle b'(\bar{X}^h_s) , b(\bar{X}^h_s)\rangle , \\
F^{b'b}_2(s)&=  \langle b'(\bar{X}^h_s), b(\bar{X}^h_s) \rangle - \langle b'(X^h_{n_s^hh}) , b(X^h_{n_s^hh})\rangle , \\
F^{b'b}_3(s)&=  \langle b'( X^h_{n_s^hh}), b( X^h_{n_s^hh} ) \rangle
-\operatorname{tr}(\Psi_{ww}(X^h_{n_s^hh},s^h,\Delta W^h_s,\Delta Z^h_s)), \\
F^{\widetilde{N}}_1(s-,z)&= \varphi^z(X_{s-})-X_{s-}- \varphi^z(\bar{X}^h_{s-})+\bar{X}^h_{s-} , \\
F^{\widetilde{N}}_2(s-,z)&=  \varphi^z(\bar{X}^h_{s-})-\bar{X}^h_{s-}- \varphi^z(X^h_{n_s^hh})+X^h_{n_s^hh} , \\
F^{\widetilde{N}}_3(s-,z)&=  \varphi^z(X^h_{n_s^hh})-X^h_{n_s^hh}\\
&- \Psi(X^h_{n_s^hh},s^h,\Delta W^h_s,\Delta Z^h_{s-}+z)+\Psi(X^h_{n_s^hh},s^h,\Delta W^h_s,\Delta Z^h_{s-}) ,  \\
F^{\nu}_1(s,z)&=  \varphi^z(X_{s})-X_{s}- \langle c(X_{s}),z\rangle
 - \varphi^z(\bar{X}^h_{s})+\bar{X}^h_{s}+ \langle c(\bar{X}^h_{s}),z\rangle, \\
F^{\nu}_2(s,z)&=  \varphi^z(\bar{X}^h_{s})-\bar{X}^h_{s}- \langle c (\bar{X}^h_{s}),z\rangle
 - \varphi^z(X^h_{n_s^hh})-X^h_{n_s^hh}+ \langle c(X^h_{n_s^hh}),z\rangle , \\
F^{\nu}_3(s,z)&=  \varphi^z(X^h_{n_s^hh})-X^h_{n_s^hh}- \langle c(X^h_{n_s^hh}),z\rangle
- \Psi(X^h_{n_s^hh},s^h,\Delta W^h_s,\Delta Z^h_{s}+z)\\
&+\Psi(X^h_{n_s^hh},s^h,\Delta W^h_s,\Delta Z^h_{s})
+ \langle \Psi_z(X^h_{n_s^hh},s^h,\Delta W^h_s,\Delta Z^h_{s}+z),z\rangle .
\ea

We calculate the expectation of each terms.
Let $0\leq t\leq T<\infty$.

\noindent
1. Consider the terms $F^a_1$, $F^{b}_1$, $F^{b'b}_1$, $F^{\widetilde{N}}_1$, $F^\nu_1$.

\noindent
1a. First, we estimate the term $F^a_1$:
\ban
\E \sup_{0\leq s\leq t}\Big|\int_0^s F^a_1(u)\,\di u\Big|^p
&\leq_C \E  \int_0^{t}|F^a_1(s)|^p\,\di s\\
&\leq_C \E\int_0^{t}|X_s-\bar{X}^h_s|^p\, \di s
\leq C_{1\text{a}} \int_0^t  \E  \sup_{0\leq u\leq s}|X_u-\bar{X}^h_u|^p \, \di s.
\ean
1b. Analogously,
\ban
\E \sup_{0\leq s\leq t}\Big|\int_0^s F^{b'b}_{1}(u)\,\di u\Big|^p
\leq
C_{1\text{b}} \int_0^t  \E  \sup_{0\leq u\leq s}|X_u-\bar{X}^h_u|^p \, \di s.
\ean
1c. To estimate the term with $F^{b}_1$ we use the moment inequality for stochastic integrals (see \cite{Mao2007}, Theorem 7.2)
\ban
\E \sup_{0\leq s\leq t}\Big|\int_0^s  \langle F^{b}_1(u), \,\di W_s\rangle\Big|^p
&\leq_C \int_0^t  \E |F^{b}_1(u)|^p\,\di s\\
&\leq C_{1\text{c}} \int_0^t  \E\sup_{0\leq u\leq s}|X_u-\bar{X}^h_u|^p \, \di s.
\ean
1d. To estimate the term with $F^{\tilde N}_1$ we use Novikov's inequality (see Theorem 4.20 in \cite{kuhn2023maximal}):
\ban
&\E \sup_{0\leq s\leq t}\Big|\int_0^s\int_{|z|>0} F^{\widetilde N}_1(u;z)\,\widetilde N(\di z,\di u)\Big|^p\\
&\leq_C \E\Big(\int_0^t\int_{|z|>0} |F^{\widetilde N}_1(u;z)|^2\, \nu(\di z)\,\di u \Big)^{p/2}
+ \E \int_0^t\int_{|z|>0} |F^{\widetilde N}_1(u;z)|^p\, \nu(\di z)\,\di u\\
&\leq_C \E \int_0^t\Big(\int_{|z|>0} |F^{\widetilde N}_1(u;z)|^2\, \nu(\di z)\Big)^{p/2}\,\di u
+  \E \int_0^t\int_{|z|>0} |F^{\widetilde N}_1(u;z)|^p\, \nu(\di z)\,\di u\\
&\leq_C \E \int_0^t\Big(\int_{|z|>0} \|\phi^z-\mathrm{Id}\|_\text{Lip}^2   | X_{u}-\bar{X}^h_{u}   |^2\, \nu(\di z)\Big)^{p/2}\,\di u\\
&+  \E \int_0^t \int_{|z|>0} \|\phi^z-\mathrm{Id}\|_\text{Lip}^p   | X_{u}-\bar{X}^h_{u}   |^p \, \nu(\di z) \,\di u\\
&\leq   \Big(\int_{|z|>0}
\|\phi^z-\mathrm{Id}\|_\text{Lip}^2 \, \nu(\di z)\Big)^{p/2}
\int_0^t  \sup_{0\leq r\leq u}\E| X_{r}-\bar{X}^h_{r}   |^p \,\di u\\
&+ \int_{|z|>0} \|\phi^z-\mathrm{Id}\|_\text{Lip}^p \, \nu(\di z)
\cdot  \int_0^t \E  \sup_{0\leq r\leq u}| X_{r}-\bar{X}^h_{r}|^p \,\di u\\
&\leq C_{1\di} \int_0^t \E  \sup_{0\leq u\leq s}| X_{u}-\bar{X}^h_{u}|^p \,\di s,
\ean
where we have used that by Lemma \ref{t:phi} (4)
\ban
\|\phi^z-\text{Id}\|_\text{Lip}\leq_C |z|\ex^{\|Dc\||z|}.
\ean
1e. Finally, we recall Lemma \ref{t:phi} (6), namely, that
\ban
\|\phi^z(\cdot)-\text{Id}-c(\cdot)z\|_\text{Lip}\leq_C |z|^2\ex^{\|Dc\||z|}.
\ean
Therefore,
\ban
\E &\sup_{0\leq s\leq t}\Big|\int_0^s\int_{|z|>0} F^{\nu}(u;z)\,\nu(\di z)\,\di u\Big|^p
\leq_C \E  \int_0^t  \Big| \int_{|z|>0} F^{\nu}(u;z)  \,\nu(\di z)\Big|^p\,\di u\\
&\leq \Big(
\int_{|z|>0} \|\phi^z(\cdot)-\text{Id}-c(\cdot)z\|_\text{Lip}\,\nu(\di z)\Big)^p
\cdot \int_0^t \E  \sup_{0\leq r\leq u}| X_{r}-\bar{X}^h_{r}|^p \,\di u\\
&= C_{1\ex}\int_0^t \E \sup_{0\leq u\leq s}| X_{u}-\bar{X}^h_{u}|^p \,\di s.
\ean

\noindent
2. Consider the terms $F^a_2$, $F^b_2$, $F^{b'b}_2$, $F^{\widetilde{N}}_2$ and $F^\nu_2$. Using the Lipschitz continuity estimates analogous to Step 1,
we get that all the expectations containing these terms are bounded from above by
\ban
\text{Const}\cdot \E  \int_0^{t}|\bar{X}^h_s-X^h_{n_sh}|^p\,\di s
= \text{Const}\cdot \E \int_0^{t}|\Psi(X^h_{n_sh},s^h,\Delta W^h_s,\Delta Z^h_s)-X^h_{n_sh}|^p\,\di s.
\ean
By means of Lemma~\ref{l:psi} (c), for each $s\in[0, T]$ we estimate
\ban
&|\Psi(X^h_{n_sh},s^h,\Delta W^h_s,\Delta Z^h_s)-X^h_{n_sh}|\\
&\qquad\qquad \qquad\leq_C(1+|X^h_{n_sh}|)( s^h+|\Delta W^h_s|+|\Delta Z^h_s|)\ex^{\|Da\|s^h+\|Db\||\Delta W^h_s|  +\|Dc\| |\Delta Z^h_s|}.
\ean
Since $\Delta W^h_s$ and $\Delta Z^h_s$ are independent of $X^h_{n_sh}$, by Lemma \ref{l:cond} we get
\ban
\E &|\Psi(X^h_{n_sh},s^h,\Delta W^h_s,\Delta Z^h_s)-X^h_{n_sh}|^p\\
&\leq_C (1+\E |X^h_{n_sh}|^p)
\cdot \E  ( s^h+|\Delta W^h_s|+|\Delta Z^h_s|)^p\ex^{p\|Da\|s^h +p\|Db\||\Delta W^h_s| +p\|Dc\| |\Delta Z^h_s|}\\
&\leq_C h (1+\E |X^h_{n_sh}|^p) ,
\ean
so that by Theorem \ref{t:existenceWZ} all the terms of the group 2 are bounded by
\ban
\text{Const}\cdot h \Big(1+ \max_{kh\leq T}\E |\bar X^h_{kh}|^p\Big) \leq \text{Const}\cdot h (1+|x|^p ) .
\ean
3. We estimate the terms of the group 3 separately.

\noindent
3a.
For $F^a_3$, since
$$
\Psi_\tau(x,0,0,0) =a(x),
$$
the Taylor expansion yields
\ban
\Psi_\tau(x,\tau,w,z)-a(x)
&= \Psi_\tau(x,\tau,w,z)- \Psi_\tau(x,0,0,0)\\
&=\int_0^1 \Psi_{\tau\tau}(x,\theta\tau,\theta w,\theta z)\tau\, \di \theta \\
&+  \int_{0}^1 \langle \Psi_{\tau w}(x,\theta\tau,\theta w,\theta z),w\rangle\, \di \theta
+ \int_0^1 \langle \Psi_{\tau z}(x,\theta\tau,\theta w,\theta z),z\rangle\, \di \theta .
\ean
Then for $s\in(n_sh,(n_s+1)h]$,
\ban
|\Psi_\tau(X^h_{n_sh},s^h,\Delta W^h_s,\Delta Z^h_s)&-a(X^h_{n_sh})|^p
\leq_C
h^p\sup_{0\leq \theta\leq 1}|\Psi_{\tau\tau}(X^h_{n_sh},\theta s^h,\theta \Delta W^h_s,\theta \Delta Z^h_s)|^p\\
&+|W^h_s|^p \sup_{0\leq \theta\leq 1} |\Psi_{\tau w}(X^h_{n_sh},\theta s^h,\theta \Delta W^h_s,\theta \Delta Z^h_s)|^p \\
&+|Z^h_s|^p \sup_{0\leq \theta\leq 1} |\Psi_{\tau z}(X^h_{n_sh},\theta s^h,\theta \Delta W^h_s,\theta \Delta Z^h_s)|^p.
\ean
By Conditions $\textbf{H}_{\Psi,K}$ and $\textbf{H}_{\nu,pK+}$ for the second derivatives we have
\ban
|\Psi_\tau&(X^h_{n_sh},s^h,  \Delta W^h_s,\Delta Z^h_s)-a(X^h_{n_sh})|^p\\
&\leq_C
(1+|X^h_{n_sh}|^p) (h^p+|\Delta W^h_s|^p+|\Delta Z^h_s|^p)\ex^{p\kappa_1 s^h+p\kappa_2|\Delta W^h_s|+p K|\Delta Z^h_s|},
\ean
what leads to the final estimate of the integral:
\ban
&\E\sup_{0\leq s\leq t}\Big|\int_0^s |\Psi_\tau(X^h_{n_uh},s^h,\Delta W^h_u,\Delta Z^h_u)-a(X^h_{n_u h})|^p\, \di u\Big|^p \\
&\leq_C
\int_0^{t}
\E \Big[(h^p+|\Delta W^h_s|^p+|\Delta Z^h_s|^p)\ex^{p\kappa_1 s^h+p\kappa_2|\Delta W^h_s|+p K|\Delta Z^h_s|}\Big] \E (1+|X^h_{n_sh}|^p)\,\di s\\
&\leq C_{3\text{a}}h(1+|x|^p).
\ean
3b, 3c. As in the estimation for $F^a_3$, we also obtain
\ban
&\E \sup_{0\leq s\leq t}\Big|\int_0^s \langle F^{b}_3(u),\, \di W_u \rangle\Big|^p
\leq C_{3\text{b}}h(1+|x|^p), \\
&\E \sup_{0\leq s\leq t}\Big|\int_0^s F^{b'b}_3(u)\, \di u\Big|^p
\leq C_{3\text{c}}h(1+|x|^p).
\ean
3d. $F^{\widetilde{N}}_3$ is estimated as follows:
\ban
\Big|\Psi(x,\tau,w,z+\zeta)-\Psi&(x,\tau,w,z)-\varphi^\zeta(x)+x \Big|\\
&=\Big|\Psi(x,\tau,w,z+\zeta)-\Psi(x,\tau,w,z)-\Psi(x,0,0,\zeta)+\Psi(x,0,0,0)\Big| \\
&\leq_C
\tau\|\zeta\|\max_{i=1,\dots,d}\sup_{\theta,\theta'\in[0,1]}|\Psi_{\tau z^i}(x,\theta''\tau,\theta''w,\theta''z+\theta' \zeta)|\\
&+\|w\|\|\zeta\|\max_{i,j=1,\dots,d}\sup_{\theta,\theta'\in[0,1]}|\Psi_{w^i z^j}(x,\theta''\tau,\theta''w,\theta''z+\theta' \zeta)|\\
&+\|z\|\|\zeta\|\max_{i,j=1,\dots,d}\sup_{\theta,\theta'\in[0,1]}|\Psi_{z^i z^j}(x,\theta''\tau,\theta''w,\theta''z+\theta' \zeta)|\\
&\leq_C (1+|x|)  (|\tau| + |w| +|z| )|\zeta| \ex^{\kappa_1 \tau + \kappa_2|w| +K(|z|+|\zeta|)}.
\ean
Hence again by Novikov's inequality,
\ban
&\E \sup_{0\leq s\leq t}\Big|\int_0^s\int_{|\zeta|>0} F^{\widetilde{N}}_3(u,\zeta)\,\widetilde N(\di \zeta,\di u)\Big|^p\\
&\leq_C \E \int_0^t \Big(\int_{|\zeta|>0}|F^{\tilde{N}}_3(u,\zeta)|^2\, \nu(\di \zeta)\Big)^{p/2}\, \di s
+ \E \int_0^t\int_{|\zeta|>0}|F^{\tilde{N}}_3(u,\zeta)|^p\, \nu(\di \zeta)\, \di s\\
&\leq_C
\Big(\Big(\int_{|\zeta|>0}|\zeta|^2 \ex^{2K|\zeta|}\, \nu(\di \zeta)\Big)^{p/2}
+ \int_{|\zeta|>0}|\zeta|^p \ex^{pK|\zeta|}\, \nu(\di \zeta) \Big)\times\\
&\times
\int_0^t \E (1+|X^h_{n_sh}|)^p \E (|s^h|^p+|\Delta W^h_s|^p+|\Delta Z^h_s|^p)
\ex^{p\kappa_1s^h+p\kappa_2|\Delta W^h_s|+pK|\Delta Z^h_s|}\,\di s\\
&\leq
C_{3\text{d}} h(1+|x|^p).
\ean
3e. Eventually we consider the term $F^{\nu}_3$. We take into account that $\Psi_{z}(x,0,0,0)^T=c(x)$ to get
\ban
\Psi&(x,\tau,w,z+\zeta)-\Psi(x,\tau,w,z)-\langle \Psi_z(x,\tau,w,z),\zeta\rangle
-\varphi^\zeta(x)+x +  \langle c(x),\zeta\rangle  \\
&\leq_C
\tau\|\zeta\|^2\max_{i=1,\dots,d}\sup_{\theta,\theta', \in[0,1]}|
\Psi_{\tau z^i z^j }(x,\theta'\tau,\theta' w,\theta' z+\theta \zeta)|\\
&+\|w\|\|\zeta\|^2\max_{i,j,k=1,\dots,d}\sup_{\theta,\theta'\in[0,1]}
|\Psi_{w^i z^j z^k}(x,\theta' \tau,\theta' w,\theta' z+\theta \zeta)|\\
&+\|z\|\|\zeta\|^2\max_{i,j,k=1,\dots,d}
\sup_{\theta,\theta'\in[0,1]}
|\Psi_{z^i z^jz^k}(x,\theta' \tau,\theta' w,\theta' z+\theta \zeta)|\\
&\leq_C (1+|x|) |\zeta|^2 (|\tau|+ |w|+|z|) \ex^{\kappa_1 \tau + \kappa_2|w| + K(|z|+|\zeta|)}
\ean
Then by Condition $\textbf{H}_{\Psi,K}$ for the third derivatives of $\Psi$ we have
\ban
&\E \sup_{0\leq s\leq t}\Big|\int_0^t\int_{|\zeta|>0} F^{\nu}_3(u)\,\nu(\di \zeta)\, \di s\Big|^p\\
&\leq_C
\Big( \int_{|\zeta|>0}|\zeta|^2 \ex^{K|\zeta|}\, \nu(\di z)\Big)^p\times\\
&\times\int_0^t \E (1+|X^h_{n_sh}|)^p
\E \Big[(|s^h|^p+|\Delta W^h_s|^p+|\Delta Z^h_s|^p)\ex^{p\kappa_1s^h+p\kappa_2|\Delta W^h_s|+pK|\Delta Z^h_s|}\Big]\,\di s\\
&\leq C_{3\text{e}} h(1+|x|^p).
\ean
Combining all the above estimates form Steps 1, 2 and 3, we get
\ban
\E\sup_{0\leq s\leq t}|X_t(x)-\bar{X}^h_t(x)|^p \leq_C h(1+|x|^p) + \int_0^t \E \sup_{0\leq u\leq s}|X_u(x)-\bar{X}^h_u(x)|^p\,\di s,
\ean
so that the statement follows from the Gronwall Lemma.

\section{Proof of Theorem $\ref{t:main}$\label{s:main proof}}

We follow here the proof of the uniform convergence of random fields as presented in Appendix in Kunita \cite{Kunita-04}.

For $T\in[0,\infty)$ fixed and $h\in(0,1]$, we consider random fields
\ban
U^h(x):= \sup_{t\in[0,T]}| X_t(x) - \bar{X}^h_t(x)|,\quad x\in\bR^d.
\ean
Let $N\in(0,\infty)$.
By Theorem \ref{thm:Xx-Xhx}, for any $p\in[2,\infty)$
\ba
\label{e:UUU}
\sup_{|x|\leq N}\E |U^h(x)|^p \leq_C h.
\ea
Furthermore, with the help of Theorems \ref{t:existenceSDE} and \ref{t:existenceWZ}
for all $h\in(0,1]$ and $|x|,|y|\leq N$
we estimate:
\ba
\label{e:h}
(\E|U^h(x)-U^h(y)|^p )^{1/p}
\leq (\E |U^h(x)|^p )^{1/p} +   ( \E |U^h(y)|^p  )^{1/p}
\leq_C h^{1/p}.
\ea
We also note that
\ban
U^h(x)-U^h(y)&=  \sup_{t\in[0,T]}| X_t(x) - \bar{X}^h_t(x)| - \sup_{t\in[0,T]}| X_t(y) - \bar{X}^h_t(y)|\\
&\leq \sup_{t\in[0,T]} (| X_t(x) - \bar{X}^h_t(x) -(  X_t(y) -\bar{X}^h_t(y)) | + |X_t(y) - \bar{X}^h_t(y) | )\\
&- \sup_{t\in[0,T]}| X_t(y) - \bar{X}^h_t(y)|\\
&\leq   \sup_{t\in[0,T]} | X_t(x) - X_t(y) | + \sup_{t\in[0,T]} |\bar{X}^h_t(x) -\bar{X}^h_t(y)|
\ean
and analogously,
\ban
U^h(y)-U^h(x)\leq   \sup_{t\in[0,T]} | X_t(x) - X_t(y) | + \sup_{t\in[0,T]} |\bar{X}^h_t(x) -\bar{X}^h_t(y)|
\ean
Hence with the help of Theorem \ref{thm:Xx-Xhx} we get that for all $h\in(0,1]$ and $|x|,|y|\leq N$
\ba
\label{e:xy}
(\E  |U^h(x)-U^h(y) |^p )^{1/p}
&\leq (  \E\sup_{0\leq t\leq T} |X_t(x) - X_t(y) |^p )^{1/p}\\
&+    ( \E\sup_{0\leq t\leq T} |\bar X_t^h(x) - \bar X_t^h(y) |^p )^{1/p}\leq_C |x-y| .
\ea
Combining the estimates \eqref{e:h} and \eqref{e:xy} we get that for each $q\in (0,1)$
\ban
 \E |U^h(x)-U^h(y)|^p )^{1/p}\leq_C  ( h^{1/p}\wedge |x-y| ) \leq_C h^{q/p} |x-y|^{1-q}.
\ean
For any $p$ that satisfies $p\geq 2$ and $p>d$
we choose $q\in (0,1-d/p)$, so that $p(1-q)>d$.
Then
\ban
\|U^h\|_{p,p(1-q)}&:=\sup_{|x|\leq N} (\E |U^h(x)|^p )^{1/p} +
\sup_{|x|,|y|\leq N, x\neq y} \frac{(\E |U^h(x) - U^h(y)|^p )^{1/p} }{|x-y|^{p(1-q)/p}}\\
&\leq_C h^{1/p}  + h^{q/p}\leq_C h^{q/p}.
\ean
Consequently, as in the proof of Theorem 4.4 in Kunita \cite{Kunita-04}, we get
\ba
\label{e:UU}
 (\E \sup_{|x|\leq N}| U^h(x)-U^h(0)  |^p  )^{1/p} \leq_C h^{q/p}.
\ea
Combining \eqref{e:UUU} and \eqref{e:UU} yields
\ban
 (\E \sup_{|x|\leq N}| U^h(x)|^p )^{1/p}
&\leq  (\E \sup_{|x|\leq N}| U^h(x)-U^h(0)  |^p )^{1/p} +  (\E |U^h(0)|^p  )^{1/p}
\leq_C h^{q/p}.
\ean
and therefore
\ban
\E \sup_{|x|\leq N}| U^h(x)| \leq_C h^{q/p}.
\ean
Now we maximize the exponent $q/p$. We choose
$p=2d$ and for $\e\in(0,1)$ we set
\ban
q=1-\frac{d}{p}(1+\e)=1- \frac{1+\e}{2}=  \frac{1-\e}{2}.
\ean
This choice yields the estimate
\ban
\E \sup_{|x|\leq N}| U^h(x)|
\leq_C h^\frac{1-\e}{4d}
\ean
and finishes the proof.

\section*{Acknowledgments}
S.\ Thipyrat was supported by the Development and Promotion of Science and Technology Talents Project (DPST), Thai Government Scholarship.


\end{document}